\documentclass [11 pt, twoside]{amsart}

\usepackage{enumerate}

 \usepackage{amsmath,amsfonts,amsthm,amssymb}
 \usepackage[all]{xy}

\newtheorem{theorem}{Theorem}[section]
\newtheorem{lemma}[theorem]{Lemma}
\newtheorem{proposition}[theorem]{Proposition}
\newtheorem{corollary}[theorem]{Corollary}

\theoremstyle{definition}
\newtheorem{definition}[theorem]{Definition}
\newtheorem{example}[theorem]{Example}

\theoremstyle{remark}
\newtheorem{remark}[theorem]{Remark}

\numberwithin{equation}{section}

\def\LX{\mathfrak L(X) }
\def\TT{\mathbb T}
\def\NN{\mathbb N}
\def\RR{\mathbb R}
\def\QQ{\mathbb Q}
\def\veps{\varepsilon}

\newcommand{\fgd}[2]{F_{#1,#2}^{\mathbf g,\mathbf d}}
\newcommand{\fuv}{\fgd{u}{v}}
\newcommand{\fvw}{\fgd{v}{w}}

\newcommand{\fuw}{\fgd{u}{w}}
\newcommand{\fuu}{\fgd{u}{u}}
\newcommand{\fxx}{\fgd{x}{x}}
\newcommand{\fxy}{\fgd{x}{y}}
\newcommand{\fyx}{\fgd{y}{x}}
\newcommand{\ldens}{\underline{\textrm{dens}}}

\begin{document}

\title[ Rotated orbits ]{ Hypercyclic operators and rotated orbits with polynomial phases }

\author{F. Bayart}
\author{G. Costakis}
\email{Frederic.Bayart@math.univ-bpclermont.fr} \email{costakis@math.uoc.gr}
\address{Clermont Universit\'e, Universit\'e Blaise Pascal, Laboratoire de Math\'ematiques, BP 10448, F-63000 CLERMONT-FERRAND ---
CNRS, UMR 6620, Laboratoire de Math\'e\-matiques, F-63177 AUBIERE}
\address{Department of Mathematics
University of Crete Knossos Avenue GR-714 09 Heraklion, Crete Greece }
\thanks{}


\date{}


\begin{abstract}
An important result of Le\'on-Saavedra and M\"uller says that the rotations of hypercyclic operators remain
hypercyclic. We provide extensions of this result for orbits of operators which are rotated by unimodular
complex numbers with polynomial phases. On the other hand, we show that this fails for unimodular complex
numbers whose phases grow to infinity too quickly, say at a geometric rate. A further consequence of our work is
a notable strengthening of a result due to Shkarin which concerns variants of Le\'on-Saavedra and M\"uller's
result in a non-linear setting.
\end{abstract}

\maketitle

\section{Introduction}
Linear dynamical systems have surprising properties. By a linear dynamical system, we mean a couple $(X,T)$
where $X$ is a (complex) topological vector space and $T\in\LX$ is a bounded linear operator on $X$.  We are interested in transitive
linear dynamical systems; namely we ask that $T$ is hypercyclic: there exists $x\in X$, called a hypercyclic
vector for $T$, such that its $T$-orbit $\{T^n x;\ n\geq 0\}$ is dense in $X$. We shall denote by $HC(T)$ the
set of hypercyclic vectors for $T$. A comprehensive treatment on linear dynamical systems is carried out in the
recent books \cite{BM}, \cite{GP}.

Linear dynamical systems have rigid properties which are not shared by general dynamical systems. For instance, Ansari
has shown in \cite{A} that, for any hypercyclic operator $T$ and for any $q\geq
1$, $T^q$ is hypercyclic and $HC(T^q)=HC(T)$. Some years later, Le\'on and M\"uller were interested in the
rotations of a hypercyclic vector. They proved in \cite{LM} that $\{T^n x;\ n\geq 1\}$ is dense in $X$ iff
$\{\lambda T^n x;\ n\geq 1,\ \lambda\in\mathbb T\}$ is dense in $X$, where $\mathbb T$ is the unit circle. In
particular, for any $\theta\in\mathbb R$, $HC(T)=HC(e^{i\theta}T)$. For further results along this line of
research we refer to \cite{BG}, \cite{CMP}, \cite{LM}, \cite{LP}, \cite{M}, \cite{S}. \medskip

To understand the influence of ``rotations" in linear dynamics it would be desirable to fully answer the
following question.\medskip

{\bf Question.} Let $T\in \LX $ be a hypercylic operator and let $x\in X$. For which sequences $(\lambda_n )\subset\TT^\NN$ do
 we have $\overline{ \{ T^nx;\ n\geq 1 \} }=X$ if and only if
$\overline{ \{ \lambda_n T^nx;\ n\geq 1 \} }=X$?\medskip

Let us first show that an affirmative answer to the above question already imposes some restrictions to the
sequence $(\lambda_n)$. Suppose $T\in \LX $ is a hypercyclic operator and consider $x\in HC(T)$. Take any
non-zero functional $x^*\in X^*$ and define $\lambda_n= |x^*(T^nx)|/x^*(T^nx)$ if $x^*(T^nx)\neq 0$ and
$\lambda_n=1$ if $x^*(T^nx)= 0$. Observe that $\lambda_n\in \mathbb{T}$ for every $n\in \mathbb{N}$. We claim
that the set $\{ \lambda_nT^nx;\ n\geq 1 \}$ is not dense in $X$. To see this, we argue by contradiction and so
assume that $\{ \lambda_nT^nx;\ n\geq 1 \}$ is dense in $X$. Then, the set $\{ x^*(\lambda_nT^nx);\ n\geq 1 \}$
is dense in $\mathbb{C}$. However, the last is a contradiction since, by the construction of $\lambda_n$,
$x^*(\lambda_nT^nx)\geq 0$ for every $n\in \mathbb{N}$.

We may also obtain a counterexample if we require moreover that $(\lambda_n)$ belongs to a finite
subset of $\mathbb T$. Indeed, let $T\in \LX$ be a hypercyclic operator on a Banach space $X$ and take $x\in HC(T)$. Then define
$\lambda_n=1$ if $\| T^nx-x\| \geq \| x\| /2$ and $\lambda_n=-1$ if $\| T^nx-x\| < \| x\| /2$. It readily
follows that $\| \lambda_nT^nx-x\| \geq \| x\| /2$ for every $n\geq 1$.
\medskip

On the contrary, the Le\'on-Saavedra and M\"uller result shows that the answer is true provided
$(\lambda_n)=(e^{in\theta})$, $\theta\in\mathbb R$. Our main result says that the rotations of the orbit by
unimodular complex numbers with linear phases can be replaced by rotations consisting of unimodular complex
numbers with polynomial phases.

\begin{theorem} \label{mainthm}
Let $X$ be a complex topological vector space, let $T\in \LX$ and let $x\in X$. The following are equivalent.
\begin{itemize}
        \item [(i)] $x$ is hypercyclic for $T$.
        \item [(ii)] $\overline{ \{ e^{iP(n)}T^n x;\ n\geq 1\}}=X$ for any polynomial $P\in\mathbb R[t]$.
        \item [(iii)] $\overline{ \{ e^{iP(n)}T^n x;\ n\geq 1\}}^o\neq \emptyset$ for some polynomial $P\in\mathbb R[t]$.
\end{itemize}
\end{theorem}
The equivalence of (ii) and (iii) in this theorem is in fact an extension of Bourdon-Feldman's theorem, which
says that $\{ T^nx;\ n\geq 0\}$ is dense if and only if it is somewhere dense, see \cite{BF}. Actually, this
result holds even for the projective orbit i.e., if the set $\{ \lambda T^nx;\ \lambda \in \mathbb{C}, n\geq 0
\}$ is somewhere dense then it is everywhere dense. In this case $T$ is called supercyclic and $x$ is called
a supercyclic vector for $T$. The supercyclic version of Bourdon-Feldman's theorem will be used in the proof of Theorem
\ref{mainthm}. That (i) implies (ii) will depend on an extension of a result of Shkarin which has its own interest.

\medskip

The problem of rotations of hypercyclic vectors can also be studied for stronger forms of hypercyclicity.
We recall that the lower density of a set of natural numbers $A$ is defined by
$$
\ldens(A):=\liminf_{N\to\infty}\frac{\textrm{card}(A\cap [1,N])}{N}\,\cdot$$ An operator $T\in\LX$ is called
frequently hypercyclic provided there exists $x\in X$ such that, for any $U\subset X$ open and nonempty, the set
$\{n\in\NN;\ T^n x\in U\}$ has positive lower density. The vector $x$ is then called frequently hypercyclic for
$T$ and the set of $T$-frequently hypercyclic vectors will be denoted by $FHC(T)$.

It has be shown in \cite{BM} that frequent hypercyclicity is invariant under rotation: for any $\lambda\in\TT$,
$FHC(T)=FHC(\lambda T)$. Here is the polynomial version of this property, to be proved in Section 3 of the paper.
\begin{theorem}\label{mainthmfhc}
Let $X$ be a complex topological vector space, let $T\in\LX$ and let $x\in X$. The following are equivalent.
\begin{itemize}
        \item [(i)] $x\in FHC(T)$.
\item[(ii)] For any $P\in\mathbb R[t]$, for any $U\subset X$ open and nonempty,
$\{n\in\mathbb N;\ e^{iP(n)}T^n x\in U\}$ has positive lower density.
\item[(iii)] There exists $P\in\mathbb R[t]$ such that, for any $U\subset X$
open and nonempty, $\{n\in\mathbb N;\ e^{iP(n)}T^n x\in U\}$ has positive lower density.
\end{itemize}
\end{theorem}

\medskip

In Section 4,
we investigate other choices of phases. We show that $\{e^{if(n)}T^n x;\ n\geq 1\}$ does not need to be dense
for all $x\in HC(T)$ provided $f$ goes sufficiently fast to infinity, for instance if $f$ has exponential growth. On the
contrary, we extend Theorem \ref{mainthm} to sequences which do not grow too quickly to infinity, for instance
for sequences like $f(n)=n^a\log^b (n+1)$. The polynomial case (Theorem \ref{mainthm}) plays a crucial role in this extension.

Finally, in Section 5, we study the link between the problem of rotations of hypercyclic vectors and the theory of uniformly distributed sequences.
In particular, we point out that the uniform distribution of $(f(n))$ is not sufficient to ensure that $\{e^{i f(n)}T^n x;\ n\geq 1\}$
is dense for any $x\in HC(T)$. Nevertheless, uniform distribution will be a useful tool to obtain generic statements.

\section{Proof of the main result and an extension of a theorem of Shkarin}
\subsection{The strategy}
As mentioned in the introduction, the equivalence between (i) and (ii) in Theorem \ref{mainthm} is already known
when $P$ is a polynomial of degree $1$. This was first done when $P(n)=2\pi n\theta$ with $\theta=\frac
pq\in\mathbb Q$, $q\geq 1$. Indeed, in that case, $e^{2\pi iqn\theta}T^{qn}x=T^{qn}x$ and Ansari has shown in
\cite{A} that $HC(T)=HC(T^q)$ for any $q\geq 1$. If $\theta$ does not belong to $\mathbb Q$, this result goes
back to the paper \cite{LM} by Le\'on-Saavedra and M\"uller: they showed that for any $\theta\in\mathbb R$,
$HC(e^{i\theta}T)=HC(T)$, namely that given $x\in HC(T)$ and any $y\in X$, there exists a sequence $(n_k)$ such
that $e^{2\pi i n_k\theta}T^{n_k}x\to y$. This was later improved by Shkarin in \cite{S} who proved the
following result (see also \cite{M} for a similar abstract result).

\begin{theorem} (Shkarin) \label{Shkarin}
Let $T$ be a hypercyclic continuous linear operator on a topological vector space $X$ and let $g$ be a generator
of a compact topological group $G$. Then $\{ (T^n x, g^n);\ n\geq 1 \}$ is dense in $X\times G$ for any $x\in
HC(T)$. \end{theorem} In particular, one may apply this result with $G=\TT$ and $g=e^{2\pi i \theta}$, which
generates $\TT$ provided $\theta\notin\mathbb Q$. This implies the Le\'on-Saavedra and M\"uller theorem since
for any $y\in X$ we may pick a sequence $(n_k)$ such that $T^{n_k}x\to y$ and $e^{2\pi in_k\theta}\to 1$.

This observation in the starting point of our investigations. We shall prove by induction a polynomial variant
of Shkarin's result. Roughly speaking, it will say that, for any polynomial $P\in\mathbb Z[t]$ and for any
$(x,y)\in HC(T)\times X$ there exists a sequence of positive integers $(n_k)$ such that $T^{n_k}x\to y$ and $g^{P(n_k)}\to
1_G$. At a first glance this seems weaker than Theorem \ref{Shkarin}, since we do not get the density of $\{
(T^n x,g^{P(n)});\ n\geq 1 \}$. And of course we can not do better because we do not require that $g$ is a
generator of $G$. This allows us to handle the case of a rational phase in the same process. Moreover, this
variant is exactly what is needed to deduce Theorem \ref{mainthm}. If we further assume that $\{ g^{P(n)};\ n\geq
1\}$ is dense in $G$, in which case $g$ is necessarily a generator of $G$, then we will be able to show the density
of $\{ (T^n x,g^{P(n)});\ n\geq 1 \}$. Thus, we will obtain the full extension of Shkarin's result in our setting.
\medskip

At this stage we need to introduce notations.
Let $G$ be an abelian compact topological group and let $p\geq 1$. For $\mathbf g=(g_1,\dots,g_p)\in G^p$, $\mathbf d=(d_1,\dots,d_p)\in\NN^p$,
and $u,v\in X$, we set
$$\fuv=\left\{h\in G;\ (v,h,1_G,\dots,1_G)\in N_{u}^{\mathbf g,\mathbf d}\right\}$$
where
$$N_{u}^{\mathbf g,\mathbf d}=\overline{\left\{(T^n u,g_1^{n^{d_1}},\dots,g_1^{n},g_2^{n^{d_2}},\dots,g_2^{n},\dots,g_p^{n^{d_p}},\dots,g_p^n);\ n\geq 1\right\}}.$$
Our variants of Shkarin theorem  read as follows.
\begin{theorem}\label{THMPREP}
Let $T$ be a hypercyclic continuous linear operator on a topological vector space $X$, let $p\geq 1$ and let $g_1,\dots,g_p$ be elements of a compact topological group $G$. Let also $\mathbf d=(d_1,\dots,d_p)\in\NN^p$.
Then for any $u\in HC(T)$ and any $v\in X$, $1_G\in \fuv$.
\end{theorem}

\begin{theorem} \label{ExtensionShkarin}
Let $T$ be a hypercyclic continuous linear operator on a topological vector space $X$. Let $P\in\mathbb
Z[t]$ and let $g\in G$ be such that  $\{ g^{P(n)};n\geq 1 \}$ is dense in $G$. Then $\{ (T^n x,
g^{P(n)});\ n\geq 1 \}$ is dense in $X\times G$ for any $x\in HC(T)$.
\end{theorem}

The proof of (iii) implies (i) in Theorem \ref{mainthm} will need the following variant of the Bourdon-Feldman theorem.

\begin{theorem} \label{THMBF}
Let $T\in\LX$ and let $x\in X$. If $\overline{ \{ \lambda T^nx;\ \lambda\in \mathbb{T}, n\geq 0 \} }^o\neq \emptyset$
then $\overline{ \{ \lambda T^nx;\ \lambda\in \mathbb{T}, n\geq 0 \} }=X$.
\end{theorem}

Before to prove Theorems \ref{THMPREP}, \ref{ExtensionShkarin} and \ref{THMBF} let us show how to deduce Theorem \ref{mainthm} from them.
For simplicity, throughout this section, we will assume that $X$ and $G$ are metrisable. The
same proofs work when $X$ or  $G$ are not metrisable, replacing everywhere sequences by nets.
To
show that (i) implies (ii), let $x\in HC(T)$, let $G=\mathbb T$ and let $P(n)=\theta_pn^p+\dots+\theta_1 n$ (we
may always assume that it has no constant term). Setting $g_k=e^{i\theta_k}$ and $d_k=k$ for $k=1,\dots,p$, we may apply Theorem
\ref{THMPREP}. In particular, given any $y\in X$, one may find a sequence $(n_k)_k$ such that
$$T^{n_k}x\to y,\ e^{in_k^p\theta_p}\to 1,\dots,e^{in_k\theta_1}\to 1.$$
This clearly implies that $e^{iP(n_k)}T^{n_k}x\to y$. It remains to
prove that (iii) implies (i). (iii) yields $\overline{ \{ \lambda T^nx;\ \lambda\in \mathbb{T}, n\geq 0 \}
}^o\neq \emptyset$ so, Theorem \ref{THMBF} gives $\overline{ \{ \lambda T^nx;\ \lambda\in \mathbb{T}, n\geq 0 \}
}=X$ and we conclude by the Le\'on-Saavedra and M\"uller theorem.

\subsection{Preparatory lemmas}
Let us turn to the proof of Theorem \ref{THMPREP}. We fix an operator $T$ acting on a topological vector space $X$
and a compact group $G$. We will need the following
elementary lemma.
\begin{lemma}\label{LEMDOUBLESUITE}
Let $g,h\in G$, $d\geq 1$, $m\geq 1$ and let $(n_k)$ be a sequence of integers such that
$$g^{n_k}\to 1_G,\ g^{n_k^2}\to 1_G,\dots, g^{n_k^{d-1}}\to 1_G,\ g^{n_k^d}\to h.$$
Then
$$g^{(n_k+m)^d}\to hg^{m^d}.$$
\end{lemma}
\begin{proof}
Write
\begin{eqnarray*}
g^{(n_k+m)^d}&=&\prod_{j=0}^d g^{n_k^j\binom djm^{d-j}}\\
&=&g^{n_k^d}\left(\prod_{j=1}^{d-1}\left(g^{n_k^j}\right)^{\binom dj m^{d-j}}\right)g^{m^d}\\
&\to&hg^{m^d}.
\end{eqnarray*}
\end{proof}

The sets $\fuv$ share some properties which are summarized below.
\begin{lemma} \label{subsem1}
Let $u,v,w\in X$, $p\geq 1$, $\mathbf g\in G^p$, $\mathbf d\in \NN^p$. The following hold.
\begin{itemize}
        \item[(i)] $\fuv$ is closed; $\fuv\subset \fgd{Tu}{Tv}$.
        \item[(ii)] $\fuv\fvw\subset\fuw$.
        \item[(iii)] Let $(v_k)\subset X$ and $(h_k)\subset G$ be such that $v_k\to v$, $h_k\to h$ and $h_k\in\fgd{u}{v_k}$. Then
        $h\in\fuv$.
\end{itemize}
\end{lemma}
\begin{proof}
(i) is trivial. (ii) follows from Lemma \ref{LEMDOUBLESUITE}. Indeed, for $\mathcal O$ an open neighbourhood
of $1_G$ in $G$ and $h\in G$, let us denote
$$\mathcal O_h=(h.\mathcal O)\times\mathcal O\times\dots\times\mathcal O\subset G^{d_1+\dots+d_p}.$$
Let $h_1\in \fuv$, $h_2\in\fvw$ and let $W\times \mathcal O_{h_1h_2}$ be an open neighborhood of the point
$(w,h_1h_2,1_G,\dots,1_G)$ in $X\times G^{d_1+\dots+d_p}$. Since $h_2\in \fvw$, there exists $m\in\NN$ such that
$$T^m v\in W\textrm{ and }(g_1^{m^{d_1}},\dots,g_p^m)\in \mathcal O_{h_2}.$$
Let $V$ be an open neighbourhood of $v$ such that $T^m V\subset W$. Since $h_1\in \fuv$, an application
of Lemma \ref{LEMDOUBLESUITE} gives the existence of an integer $n$ satisfying
\begin{itemize}
\item $T^n u\in V\implies T^{n+m}u\in W$;
\item $g_1^{(n+m)^{d_1}}\in h_1h_2\mathcal O$;
\item $g_l^{(n+m)^k}\in\mathcal O$ provided $l=1$ and $k\leq d_1-1$ or $l\geq 2$ and $k\leq d_l$.
\end{itemize}
This shows that $h_1h_2\in \fuw$.

The proof of (iii) goes along the same lines. Let $V\times\mathcal O_h$ be an open neighbourhood of
$(v,h,1_G,\dots,1_G)$ in $X\times G^{d_1+\dots+d_p}$. There exists $k\geq 1$ such that $(v_k,h_k,1_G,\dots,1_G)\in V\times \mathcal O_h$ and
since $h_k\in\fgd{u}{v_k}$, there exists $n\geq 1$ such that $(T^n u,g_1^{n^{d_1}},\dots,g_p)\in V\times\mathcal O_h$. Thus,
$h\in\fuv$.
\end{proof}

For a proof of the following lemma see, for instance, \cite{HR}.
\begin{lemma} \label{hr}
A closed subsemigroup of a compact topological group is a subgroup.
\end{lemma}

\subsection{Proof of Theorem \ref{THMPREP}}
We are now ready for the proof of Theorem \ref{THMPREP}. We proceed by induction on $d_1+\dots+d_p$. We first
assume that $d_1+\dots+d_p=1$ and let $g\in G$, $u\in HC(T)$, $v\in X$. Define $G_0=\overline{\{g^n;\ n\geq
0\}}$. $G_0$ is an abelian compact topological group and $g$ is a generator of $G_0$. By applying Shkarin's
result, $\{(T^n u,g^n);\ n\geq 1\}$ is dense in $X\times G_0$. In particular, $1_G\in \fuv$.

Suppose now that $d_1+\dots+d_p\geq 2$ and let $u\in HC(T)$, $v\in X$. We set $\mathbf d'=(d_1-1,\dots,d_p)$. We consider
any $x,y\in HC(T)$. By the induction hypothesis,
$1_G\in F_{x,y}^{\mathbf g,\mathbf d'}$. This yields, by compactness of $G$, that $\fxy$ is nonempty. In particular, Lemma \ref{subsem1}
tells us that $\fxx$ is a closed subsemigroup of $G$, hence a closed subgroup of $G$. Moreover, if we use again Lemma \ref{subsem1},
we observe that
$$\left\{
\begin{array}{rcl}
\fxx\fxy&\subset&\fxy\\[2mm]
\fxy\fyx&\subset&\fxx.
\end{array}\right.$$
Since $\fxy$ and $\fyx$ are both nonempty, $\fxy$ contains a coset of $\fxx$ and is contained in a coset of $\fxx$. Thus it is a coset of $\fxx$
(at this point, it is important to notice that we need that $G$ is abelian).

We apply this for $x=u$ and $y=Tu$: there exists $g\in G$ such that $\fgd{u}{Tu}=g\fuu$. Now, using again (i) and (ii) of Lemma \ref{subsem1}, we get
$$ \left\{
\begin{array}{l}
 \fgd{Tu}{T^2u}\supset \fgd{u}{Tu}\supset g\fuu\\[2mm]
\fgd{u}{T^2 u}\supset \fgd{u}{Tu}\fgd{Tu}{T^2 u}\supset g^2\fuu.
\end{array}\right.$$
Since $\fgd{u}{T^2 u}$ is a coset of $\fuu$, this in turn yields $\fgd{u}{T^2u}=g^2\fuu$.
Repeating the process, we obtain that, for any $n\geq 1$,
$\fgd u{T^n u}=g^n \fuu.$
Now, we use again the induction hypothesis, but for $d_1+\dots+d_p=1$. This gives a sequence $(n_k)$ such that $g^{n_k}\to 1_G$ and
$T^{n_k}u\to v$. By the last part of Lemma \ref{subsem1}, $1_G\in \fuv$, which achieves the proof of Theorem \ref{THMPREP}.

\subsection{Proof of Theorem \ref{ExtensionShkarin}}

Fix $x\in HC(T)$ and let $y\in X$, $m\in \mathbb{N}$. Let also $d$ be the degree of the polynomial $P$,
$d\geq 1$. Then $T^mx\in HC(T)$ and by Theorem \ref{THMPREP} there exists a sequence of positive integers
$(n_k)$ such that
$$T^{n_k}(T^mx)\to y,\ g^{n_k}\to 1_G,\ g^{n_k^2}\to 1_G, \ldots,\ g^{n_k^d} \to 1_G.$$ From the last we deduce
that $$T^{n_k+m}x\to y, \ g^{(n_k+m)^j}\to g^{m^j} \,\,\, \textrm{for every}\,\,\, j=0,1,\ldots ,d$$ and this in
turn implies $T^{n_k+m}x\to y$, $g^{P(n_k+m)}\to g^{P(m)}$. Thus,
$$(y, g^{P(m)})\in \overline{ \{ (T^n x, g^{P(n)});\ n\geq 1 \} } \ \textrm{for every pair} \ (y,m)\in X \times \mathbb{N}.$$
Since $\{ g^{P(m)};m\geq 1 \}$ is dense in $G$ the conclusion follows.

\subsection{An extension of the Bourdon-Feldman result}
The next series of lemmas will be used in the proof of Theorem \ref{THMBF}. This kind of approach has appeared
in \cite{CH1} and borrows ideas from \cite{P}.

\begin{lemma} \label{l1BF}
Let $x,y$ be vectors in $X$. If $$\overline{\{ \lambda T^nx;\ \lambda \in \mathbb{T}, n\geq 0\} }^o\cap
\overline{ \{ \lambda T^ny;\ \lambda \in \mathbb{T}, n\geq 0\}   }^o\neq \emptyset $$ then $$\overline{\{ \lambda
T^nx;\ \lambda \in \mathbb{T}, n\geq 0\}  }^o=\overline{\{ \lambda T^ny;\ \lambda \in \mathbb{T}, n\geq 0\} }^o.$$
\end{lemma}
\begin{proof}
There exist $\alpha \in \mathbb{T}$ and a positive integer $k$ such that
$$\alpha T^kx\in \overline{\{ \lambda T^nx;\ \lambda \in \mathbb{T}, n\geq 0\} }^o\cap
\overline{ \{ \lambda T^ny;\ \lambda \in \mathbb{T}, n\geq 0\}   }^o.$$ From the last we deduce that $\overline{
\{ \lambda T^nx;\ \lambda \in \mathbb{T}, n\geq k \} }\subset \overline{ \{ \lambda T^ny;\ \lambda \in \mathbb{T},
n\geq 0\}  }$ and since $ \overline{ \{ \lambda T^nx;\ \lambda \in \mathbb{T}, n\leq  k  \} }^o=\emptyset$ the
inclusion
$$\overline{ \{ \lambda T^nx;\ \lambda\in \mathbb{T}, n\geq 0 \} }^o \subset \overline{ \{ \lambda T^ny;\ \lambda \in \mathbb{T}, n\geq 0\}  }^o$$
follows. Interchanging the roles of $x$ and $y$ in the previous argument we conclude the reverse inclusion and
we are done.
\end{proof}

\begin{lemma} \label{l2BF}
Let $x\in X$. For every non-zero
complex number $\mu$, $\overline{\{ \lambda T^nx;\ \lambda\in \mathbb{T}, n\geq 0\}
 }^o=\overline{ \{ \lambda T^n(\mu x);\ \lambda\in \mathbb{T}, n\geq 0\} }^o$.
\end{lemma}
\begin{proof}
We first assume that $\overline{ \{ \lambda T^n(x);\lambda \in \mathbb{T}, n\geq
0 \} }^o\neq\emptyset$ and let $U$ be a nonempty open subset of $X$ such that $U\subset \overline{ \{ \lambda T^nx;\ \lambda\in \mathbb{T}, n\geq 0\} }^o$.
There exist a complex number $\rho $
and a positive integer $m$ such that
\begin{equation}
\rho U\cap U\neq\emptyset
\end{equation}
and
\begin{equation}
{\rho}^m=\mu .
\end{equation}
The inclusion $\rho^l U\subset \overline{ \{ \lambda T^n({\rho}^lx);\lambda \in \mathbb{T}, n\geq
0 \} }^o$ trivially holds for every non-negative integer $l$. Because of (2.1) we get
\[
\rho^l U\cap\rho^{l+1}U  \neq \emptyset\,\, \textrm{for every}\,\,l=0,1,2,\ldots .
\]
Hence, for every $l=0,1,\ldots ,m-1$
\[
\overline{ \{ \lambda T^n({\rho}^lx);\lambda \in \mathbb{T}, n\geq 0 \} }^o\cap \overline{ \{ \lambda
T^n({\rho}^{l+1}x);\lambda \in \mathbb{T}, n\geq 0 \}}^o \neq\emptyset
\]
and by (2.2) and Lemma \ref{l1BF} we conclude that $$ \overline{\{ \lambda T^nx;\ \lambda\in \mathbb{T}, n\geq
0\}
 }^o=\overline{ \{ \lambda T^n(\mu x);\ \lambda\in \mathbb{T}, n\geq 0\} }^o. $$

If we now assume that $\overline{ \{ \lambda T^n(x);\lambda \in \mathbb{T}, n\geq
0 \} }^o=\emptyset$, then we shall have $$\overline{ \{ \lambda T^n(\mu x);\lambda \in \mathbb{T}, n\geq
0 \} }^o=\emptyset$$ for any $\mu\neq 0$, otherwise
$$\overline{ \{ \lambda T^n(x);\lambda \in \mathbb{T}, n\geq
0 \} }^o=\overline{ \left\{ \lambda T^n\left(\frac1\mu \mu x\right);\lambda \in \mathbb{T}, n\geq 0 \right\}
}^o$$ would be nonempty.
\end{proof}

\begin{proof}[Proof of Theorem \ref{THMBF}]
Since $\overline{ \{ \lambda T^nx;\ \lambda\in \mathbb{T},n\geq 0 \} }^o\neq \emptyset$, Bourdon-Feldman's
theorem implies that $T$ is supercyclic, in fact $x$ is a supercyclic vector for $T$. By the density of
supercyclic vectors there exists a supercyclic vector $z$ for $T$ such that $z\in \overline{ \{ \lambda T^nx;
\lambda\in \mathbb{T},n\geq 0 \} }^o$. Applying Lemma \ref{l2BF} we get $\mu z\in \overline{ \{ \lambda T^nx;
\lambda\in \mathbb{T},n\geq 0 \}}^o$ for every non-zero complex number $\mu$. Since $z$ is supercyclic for $T$
and the set $\overline{ \{ \lambda T^nx;\ \lambda\in \mathbb{T},n\geq 0 \}}$ is $T$-invariant, the result
follows.
\end{proof}

\subsection{Consequences and remarks}

Theorem \ref{Shkarin} is a particular case of a more general result which can be found in \cite{S}. Let us recall that a continuous map $T:X\to X$,
where $X$ is a topological space, is universal provided there exists $x\in X$, called universal vector for $T$, such that $\{T^n x;\ n\geq 1\}$ is
dense in $X$. We denote by $\mathcal U(T)$ the set of universal vectors for $T$.

Theorem \ref{Shkarin} can be extended to nonlinear dynamical systems whose set of universal vectors satisfies connectedness assumptions. Precisely,  Shkarin has proved the following result:
\begin{quote}
Let $X$ be a topological space, let $T:X\to X$ be a continuous map and let $g$ be a generator of a compact topological group $G$.
Assume also that there is a nonempty subset $Y$ of $\mathcal U(T)$ such that $T(Y)\subset Y$ and $Y$ is path connected, locally path connected
and simply connected. Then the set $\{(T^n x,g^n);\ n\geq 1\}$ is dense in $X\times G$ for any $x\in Y$.
\end{quote}
Starting from this result and with exactly the same proof, we can get the following statement.
\begin{quote}
Let $X$ be a topological space, let $T:X\to X$ be a continuous map, let $g_1,\dots,g_p$ be elements of a compact topological group $G$ and let  $\mathbf d=(d_1,\dots,d_p)\in\NN^p$.
Assume also that there is a nonempty subset $Y$ of $\mathcal U(T)$ such that $T(Y)\subset Y$ and $Y$ is path connected, locally path connected
and simply connected. Then $X\times\{(1_G,\dots,1_G)\}\subset N_u^{\mathbf g,\mathbf d}$ for any  $u\in Y$.
\end{quote}

We shall need later a variant of Theorem \ref{ExtensionShkarin}, where we allow the use of several polynomials.
\begin{corollary}\label{COREXTENSION}
 Let $P_1,\dots,P_r$ be real polynomials and let $E$ be the closure of the set $\{e^{iP_1(n)},\dots,e^{iP_r(n)});\ n\geq1\}$.
Let $T\in\LX$ be hypercyclic and let $x\in HC(T)$. Then $\{(T^n x,e^{iP_1(n)},\dots,e^{iP_r(n)});\ n\geq 1\}$
is dense in $X\times E$.
\end{corollary}
\begin{proof}
 Let $d=\max(\deg(P_1),\dots,\deg(P_r))$ and let us write $P_p(x)=\sum_{j=0}^d \theta_{j,p}x^j$
for $p=1,\dots,r$. Let $y\in X$, $m\geq 1$ and let $(n_k)$ be a sequence of integers such that
$$T^{n_k}(T^m x)\to y,\ e^{in_{k}^l \theta_{j,p}}\to 1\textrm{ for }1\leq j,l\leq d,\ 1\leq p\leq r.$$
Then, as in the proof of Theorem \ref{ExtensionShkarin}, $T^{n_k+m}\to y$ and $e^{iP_p(n_k+m)}\to e^{iP_p(m)}$
for any $p$ in $\{1,\dots,r\}$.
\end{proof}

This corollary is particularly interesting when the closure of $ \{(e^{iP_1(n)},\dots,e^{iP_r(n)});\ n\geq 1\}$
is equal to $\TT^r$. This a well-known phenomenon in the theory of uniformly distributed sequences.

\begin{definition}
 We say that the real polynomials $P_1,\dots,P_r$ are $\pi\mathbb Q$-independent provided for any $h\in\mathbb Z^r$,
$h\neq 0$, the polynomial $h_1P_1+\dots+h_rP_r$ does not belong to $\pi\QQ[t]$.
\end{definition}

A fundamental theorem in the theory of uniformly distributed sequences, see \cite{KN},
says that if $P_1,\dots,P_r$ is a $\pi\QQ$-independent family of real polynomials, then the sequence
$(P_1(n),\dots,P_r(n))$ is uniformly distributed modulo 1. Hence, $\{(e^{iP_1(n)},\dots,e^{iP_r(n)});\ n\geq 1\}$
is dense in $\TT^r$.

\begin{corollary}
 Let $P_1,\dots,P_r$ be real polynomials which are $\pi\QQ$-independent. Let $T\in\LX$ be hypercyclic
and let $x\in HC(T)$. Then $\{(T^n x,e^{iP_1(n)},\dots,e^{iP_r(n)});\ n\geq 1\}$
is dense in $X\times\TT^r$.
\end{corollary}

\section{Rotations of frequently hypercyclic vectors}

This section is devoted to the proof of Theorem \ref{mainthmfhc}. We first need an elementary lemma on sets
with positive lower density. Its proof can be found e.g. in \cite[Lemma 6.29]{BM}.
\begin{lemma}\label{LEMROTFHC} Let $A\subset\NN$ have positive lower density. Let also $I_1,\dots, I_q\subset\NN$ with
$\bigcup_{j=1}^q I_j=\NN$, and let $n_1,\dots ,n_q\in\NN$. Then $B:=\bigcup_{j=1}^q
\left( n_j+A\cap I_j\right)$ has positive lower density.
\end{lemma}

As recalled before, if $P_1,\dots,P_r$ is a $\pi\QQ$-independent family of real polynomials, then $\{(e^{iP_1(n)},\dots,e^{iP_r(n)});\ n\geq 1 \}$
is dense in $\TT^r$. We need a variant of this result when there exist relations between the polynomials.

\begin{proposition}
 Let $P_1,\dots,P_r$ be real polynomials without constant term. Assume that there exist $p\in\{1,\dots,r\}$,
integers $m$ and $(a_{j,k})_{\substack{p+1\leq j\leq r\\ 1\leq k\leq p}}$, polynomials $(R_j)_{p+1\leq j\leq r}$ in
$\pi\mathbb Z[t]$ so that
\begin{itemize}
 \item[(i)] $P_1,\dots,P_p$ are $\pi\QQ$-independent;
\item[(ii)] For any $j$ in $\{p+1,\dots,r\}$,
$$mP_j=a_{j,1}P_1+\dots+a_{j,p}P_p+R_j.$$
\end{itemize}
Then the closure of $\{(e^{iP_1(n)},\dots,e^{iP_r(n)});\ n\geq 0\}$ is equal to
$$\left\{(e^{i\theta_1},\dots,e^{i\theta_r});\ \forall j\geq p+1,\ m\theta_j=a_{j,1}\theta_1+\dots+a_{j,p}\theta_p\right\}.$$
\end{proposition}
\begin{proof}
 Let $(\theta_1,\dots,\theta_r)\in\RR^r$ be such that $ m\theta_j=a_{j,1}\theta_1+\dots+a_{j,p}\theta_p$ for $j\geq p+1$. We define $Q_j$ and
$T_j$ by
\begin{eqnarray*}
 Q_j(x)&=&\left\{
\begin{array}{ll}
 \frac1mP_j(2mx)&\textrm{ if }j\leq p\\
P_j(2mx)&\textrm{ if }j\geq p+1\\
\end{array}\right.
\\
T_j(x)&=&\frac1mR_j(2mx),\ j\geq p+1.
\end{eqnarray*}
We may observe that $T_j(n) \in2\pi\mathbb Z$ for any integer $n$ since $R_j$ has no constant term and belongs
to $\pi\mathbb Z[t]$. It is also easy to check that the family $(Q_1,\dots,Q_p)$
remains $\pi\QQ$-independent. Hence we can find a sequence of integers $(n_k)$ such that
$e^{iQ_j(n_k)}$ goes to $e^{i\theta_j/m}$ for any $j$ in $\{1,\dots,p\}$. Now, for $j\geq p+1$,
$$Q_j=a_{j,1}Q_1+\dots+a_{j,p}Q_p+T_j$$
so that $e^{iQ_j(n_k)}$ goes to $e^{i(a_{j,1}\theta_1+\dots+a_{j,p}\theta_p)/m}=e^{i\theta_j}.$
Finally, we have shown that for any $j\leq r$, $e^{iP_j(2mn_k)}$ converges to $e^{i\theta_j}$, which implies the proposition.
\end{proof}

We shall use this proposition under the form of the following corollary.
\begin{corollary}\label{CORINVROT}
 Let $P_1,\dots,P_r$ be real polynomials without constant term. Then the closure of $\{(e^{iP_1(n)},\dots,e^{iP_r(n)});\ n\geq 1\}$
is invariant under complex conjugation.
\end{corollary}
\begin{proof}
 We extract from $(P_1,\dots,P_r)$ a maximal family $(P_j)_{j\in J}$ which is $\pi\mathbb Q$-independent.
Without loss of generality, we may assume that $J=\{1,\dots,p\}$. This means that the assumptions of
the previous proposition are satisfied. Hence, the result of the proposition describes
the closure of $\{(e^{iP_1(n)},\dots,e^{iP_r(n)});\ n\geq 1\}$. And this closure is clearly invariant under complex conjugation.
\end{proof}

We turn to the proof of Theorem \ref{mainthmfhc}. We shall in fact prove a variant of it which looks significantly
stronger, since we control simultaneously several rotated orbits.
\begin{theorem}\label{THMFHCSTRONG}
Let $P_1,\dots,P_r$ be real polynomials without constant terms.
Let also $T\in\LX$ be frequently hypercyclic and let $x\in FHC(T)$. Then, for any nonempty open set
$U\subset X$,
$$\big\{n\in\mathbb N;\ \forall l\in\{1,\dots,r\},\ e^{iP_l(n)}T^n x\in U \big\}$$
has positive lower density.
\end{theorem}
\begin{proof}
 We denote by $d$ the maximum of the degree of $P_1,\dots,P_r$ and we argue by induction on $d$. The case $d=0$ is trivial since the polynomials have to be equal to zero.
So, let us assume that the theorem has been proved until rank $d-1$ and let us prove it for $d\geq 1$. Let $V\subset X$
and let $\veps>0$ be such that $V$ is open and nonempty and $D(1,\veps)V\subset U$, where $D(a,\veps)$ means
the disk $|z-1|<\varepsilon$. Let us set
$$E=\overline{\big\{e^{iP_1(k)},\dots,e^{iP_r(k)});\ k\geq 0\big\}}\subset\TT^r.$$
By the compactness of $E$, there exist integers $m_1,\dots,m_q$ such that, for any $k\geq 0$,
one may find $j\in\{1,\dots,q\}$ so that $|e^{iP_l(k)}-e^{iP_l(m_j)}|<\veps$ for any $l=1,\dots,r$.
We then set, for $j=1,\dots,q$,
$$I_j=\big\{k\geq 0;\ \forall l\in\{1,\dots,r\},\ |e^{iP_l(k)}-e^{iP_l(m_j)}|<\veps\big\}.$$
Therefore, $\bigcup_{j=1}^q I_j=\mathbb N$.

By Corollary \ref{CORINVROT}, $E$ is invariant under complex conjugation. In particular, for any $j=1,\dots,q$,
$(e^{-iP_1(m_j)},\dots,e^{-iP_r(m_j)})$ belongs to $E$. We now apply Corollary \ref{COREXTENSION}. For any $j=1,\dots,q$,
we may find an integer $n_j$ such that, for any $l\in\{1,\dots,r\}$,
$$e^{iP_l(m_j)}e^{iP_l(n_j)}T^{n_j}x\in V.$$
Since $T$ is continuous, there exists an open neighbourhood $W$ of $x$ such that
$$\forall j\in\{1,\dots,q\},\ \forall l\in\{1,\dots,r\},\ e^{iP_l(m_j)}e^{iP_l(n_j)}T^{n_j}W\subset V.$$
Now, there exist a sequence of polynomials $(Q_{j,l})_{\substack{1\leq j\leq q\\ 1\leq l\leq r}}$ with degree at most $d-1$
and without constant term such that, for any $j$ in $1,\dots,q$, for any $l$ in $1,\dots,r$, for any $k\geq 0$,
$$P_l(n_j+k)=P_l(n_j)+P_l(k)+Q_{l,j}(k).$$
We set
$$A=\big\{k\geq 0;\ \forall (l,j)\in\{1,\dots,r\}\times\{1,\dots,q\},\ e^{iQ_{l,j}(k)}T^k x\in W\big\}.$$
By the induction hypothesis, $A$ has positive lower density. By Lemma \ref{LEMROTFHC},
this remains true for
$$B=\bigcup_{j=1}^q (n_j+A\cap I_j).$$
Now, pick any $n\in B$. There exist $j$ in $\{1,\dots,q\}$ and $k\in A\cap I_j$ such that $n=n_j+k$. This leads to
$$e^{iP_l(n_j+k)}T^{n_j+k}(x)=\underbrace{\underbrace{e^{iP_l(k)}e^{-iP_l(m_j)}}_{\in D(1,\veps)}\underbrace{e^{iP_l(m_j)}e^{iP_l(n_j)}T^{n_j}(\underbrace{e^{iQ_{j,l}(k)}T^k x)}_{\in W}}_{\in V}}_{\in U}.$$
Thus, $B\subset \big\{n\in\mathbb N;\ \forall l\in\{1,\dots,r\},\ e^{iP_l(n)}T^n x\in U\big\}.$
\end{proof}

\begin{proof}[Proof of Theorem \ref{mainthmfhc}]
That $(i)$ implies $(ii)$ is a direct consequence of the previous theorem: when we have a single polynomial,
we may allow a constant term since $e^{i\theta}x\in FHC(T)$ iff $x\in FHC(T)$.

It remains to prove $(iii)$ implies $(i)$. The proof follows the same lines; let $U,V\subset X$ be open and nonempty and let $\veps>0$
with $D(1,\veps)V\subset U$. Let $\lambda_1,\dots,\lambda_q\in\TT$ be such that $\TT$ is contained in $\bigcup_{j=1}^q D(\lambda_j^{-1},\veps)$.
For $j=1,\dots,q$, let us set
$$I_j=\big\{k\geq 0;\ e^{iP(k)}\in D(\lambda_j^{-1},\veps)\big\}.$$
Moreover, for any $j$ in $\{1,\dots q\}$, one may find an integer $n_j$ such that
$T^{n_j}x\in \lambda_jV$ since the assumption and Theorem \ref{mainthm} imply $x\in HC(T)$. Let now $W\subset X$
open and nonempty be such that $T^{n_j}(W)\subset\lambda_j V$ for any $j\in\{1,\dots,q\}$. We finally set
\begin{eqnarray*}
 A&=&\big\{k\geq 0;\ e^{iP(k)}T^k x\in W\}\\
B&=&\bigcup_{j=1}^q (n_j+A\cap I_j).
\end{eqnarray*}
$A$, thus $B$, have positive lower density. Moreover, if $n=n_j+k$ belongs to $B$, then
$$T^{n_j+k}(x)=\underbrace{\underbrace{e^{-iP(k)}}_{\in D(\lambda_j^{-1},\veps)}\underbrace{T^{n_j}\underbrace{(e^{iP(k)}T^k x)}_{\in W}}_{\in \lambda_jV}}_{\in U}.$$
\end{proof}

\begin{remark}
 Even if the result of Theorem \ref{THMFHCSTRONG} looks stronger than condition (ii) of Theorem \ref{mainthmfhc}, it is the natural statement which comes from our proof.
 Indeed, if you follow the proof for a single polynomial, then you have to apply the induction hypothesis for several polynomials!
\end{remark}

\section{Other sequences}

In this section, we study if Theorem \ref{mainthm} remains true when we replace the sequence $(P(n))$ by other classical sequences. We first show that this is not the case
if the sequence grows to infinity too quickly. The counterexample is very easy. It is just a
backward shift on $\ell^2(\mathbb Z_+)$. Wy denote by $B$ the
unweighted backward shift. As usual $(e_n)_{n\geq 0}$ denotes the standard basis on $\ell^2(\mathbb Z_+)$.
\begin{proposition} \label{geomrate}
Let $T=2B$ acting on $\ell^2(\mathbb Z_+)$ and let $(f(n))$ be a sequence of positive integers with $f(n+1)>af(n)$, $n\in \mathbb{N}$, for some $a>1$. Then for
every $x\in HC(T)$ there exists $\theta\in\mathbb R$ such that the set $\big\{ e^{2\pi if(n)\theta}T^n x;\ n\geq
1\big\} $ is not dense in $\ell^2(\mathbb Z_+)$.
\end{proposition}
\begin{proof}
Let $x\in HC(T)$. Choose $N\geq 1$ such that $\sum_{j\geq N+1}a^{-j}\leq 1/4$ and define the set
$$A=\big\{n\in\mathbb N;\ \exists \lambda\in\mathbb T,\ \|\lambda T^n x-e_N\|<1/2\big\}.$$
Observe that if $n$ belongs to $A$, then $n+k$ does not belong to $A$ for any $k\leq N$. Indeed, one can write
$T^n x=\mu e_N+y$ with $\|y\|\leq 1/2$ and $|\mu|=1$, so that
$T^{n+k}x=\mu 2^k e_{N-k}+T^k y$. This yields, for any $\lambda\in\mathbb T$,
\begin{eqnarray*}
\|\lambda T^{n+k}x-e_N\|&=&\|\lambda\mu 2^k e_{N-k}-e_N+T^k y\|\\
&\geq&2^k-2^{k-1}\geq 1.
\end{eqnarray*}
We define a sequence $(\alpha_n)$ as follows:
\begin{itemize}
\item $\alpha_n=0$ provided $n\notin A$;
\item if $n\in A$, we set $\theta_{n-1}=\sum_{j=0}^{n-1} \frac{\alpha_j}{f(j)}$. We then choose $\alpha_n\in \{0,1/2\}$ such that
$$\Re e\left(e^{2\pi i(f(n)\theta_{n-1}+\alpha_n)}e_N^*(T^n x)\right)\leq 0.$$
\end{itemize}
We finally define $\theta=\sum_{n\geq 0}\frac{\alpha_n}{f(n)}$ and we claim that $\{e^{2\pi if(n)\theta}T^n x;\ n\geq 1\}$ is not dense in $\ell^2(\mathbb Z_+)$.
Precisely, let us show that $e_N$ does not belong to the
closure of $\{e^{2\pi if(n)\theta}T^nx;\ n\geq 1\}$. Indeed, when $n$ does not belong to $A$, we are sure that
$\|e^{2\pi if(n)\theta}T^n x -e_N\|\geq 1/2$. Otherwise, when $n$ belongs to $A$, we can write
$$\|e^{2\pi if(n)\theta}T^n x-e_N\|\geq |e_N^*(e^{2\pi if(n)\theta}T^n x)-1|.$$
Now,
\begin{eqnarray*}
e_N^*(e^{2\pi if(n)\theta}T^n x)&=&e^{2\pi i(f(n)\theta_{n-1}+\alpha_n)}e_N^*(T^n x)e^{2\pi i f(n)\left(\sum_{j\geq n+1}\frac {\alpha_j}{f(j)}\right)}\\
&=&e^{2\pi i(f(n)\theta_{n-1}+\alpha_n)}e_N^*(T^n x)e^{2\pi i f(n)\left(\sum_{j\geq n+N+1}\frac {\alpha_j}{f(j)}\right)}.
\end{eqnarray*}
We set $z=e^{2\pi i(f(n)\theta_{n-1}+\alpha_n)}e_N^*(T^n x)$ and $\beta=2\pi  f(n)\left(\sum_{j\geq n+N+1}\frac {\alpha_j}{f(j)}\right)$.
By the construction of $\alpha_n$, $\Re e(z)\leq 0$. Moreover, it is easy to check that $\beta\in[0,\pi/4]$:
$$0\leq2\pi  f(n)\left(\sum_{j\geq n+N+1}\frac {\alpha_j}{f(j)}\right) \leq \pi f(n)\sum_{j\geq N+n+1}\frac{1}{a^{j-n}f(n)}\leq \frac \pi4.$$
Thus, $e_N^*(e^{2\pi if(n)\theta}T^n x)$ does not belong to the cone $\{\rho e^{i\gamma};\rho>0,\
|\gamma|\leq\pi/4\}.$ In particular, there exists $\delta>0$ such that
$$\|e^{2\pi if(n)\theta}T^n x-e_N\|\geq\delta.$$
\end{proof}

\begin{remark}
In the previous proposition one cannot conclude the stronger assertion that there exists $\theta\in\mathbb R$
such that for every $x\in HC(2B)$ the set $\{ e^{2\pi if(n)\theta}(2B)^n x;\ n\geq 1\} $ is not dense in
$\ell^2(\mathbb Z_+)$. The reason for this is very simple and comes from the fact that $2B$ is hypercyclic in a
very strong sense, namely it satisfies the hypercyclicity criterion. For a comprehensive discussion on the
hypercyclicity criterion and its several equivalent forms we refer to \cite{BM}, \cite{GP}. To explain now
briefly, take any operator $S\in \LX$ satisfying the hypercyclicity criterion and let $(\lambda_n)$ be any
sequence of unimodular complex numbers. It is then immediate that the sequence of operators $(\lambda_nS^n)$
also satisfies the hypercyclicity criterion;\ hence, the set of $y\in X$ such that $\{ \lambda_nS^ny;n\geq 1\}$
is dense in $X$ is $G_{\delta}$ and dense in $X$. By an appeal of Baire's category theorem one can find $x\in
HC(S)$ such that $\overline{ \{ \lambda_nS^nx;\ n\geq 0\} }=X$. Applying the last for $S:=2B\in \mathfrak
L(\ell^2(\mathbb Z_+))$ and $\lambda_n:=e^{2\pi if(n)\theta}$, $n\geq 1$ we see that for every $\theta\in\mathbb
R$ there exists $x\in HC(2B)$ such that the set $\{ e^{2\pi if(n)\theta}(2B)^n x;\ n\geq 1\} $ is dense in
$\ell^2(\mathbb Z_+)$.
\end{remark}

On the contrary, we have an analog to Theorem \ref{mainthm} for sequences which grow slowly to infinity. The growth condition
which comes into play here is based on the increases of the function.

\begin{theorem} \label{THMSLOWGROWTH}
Let $X$ be a Banach space, let $T\in \LX $, let $x\in X$ and let $(f(n))$ be a sequence of real numbers satisfying the following condition: there exist
an integer $d\geq 0$, sequences $(g_l(n))_n$ for $0\leq l\leq d$ and $(\veps_k(n))_n$ for any $k\geq 1$, such that, for any $n,k\geq 1$,
$$f(n+k)-f(n)=\sum_{l=0}^{d}g_l(n)k^l+\veps_k(n)$$
and, for a fixed $k\geq 1$, $|\veps_k(n)|\xrightarrow{n\to+\infty}0.$
Then the following are equivalent.
\begin{itemize}
\item[(i)]$x\in HC(T)$;
\item[(ii)]$\{e^{ i f(n)}T^n x;\ n\geq 1\}$ is dense in $X$.
\end{itemize}
\end{theorem}
\begin{proof}
We just need to prove that $(i)\implies (ii)$. We are going to apply Theorem \ref{mainthm} (observe that a polynomial $P$
satisfies the assumptions of Theorem \ref{THMSLOWGROWTH} with $d=\deg(P)$ and $\veps_k(n)=0$!) and a compactness argument.
\begin{lemma}\label{LEMSLOWGROWTH}
 Let $d\geq 0$, $x\in HC(T)$, $y\in X$ and $\veps>0$. There exists $K\geq 1$ such that, for any $P\in\mathbb R_d[t]$,
for any $\mu\in\TT$, there exists $k\leq K$ such that
\begin{eqnarray}
 \|e^{iP(k)}T^k x-\mu y\|<\veps. \label{EQLEMSLOWGROwTH}
\end{eqnarray}
\end{lemma}
\begin{proof}[Proof of Lemma \ref{LEMSLOWGROWTH}]
Let us observe that if $(P,\mu)$ satisfies (\ref{EQLEMSLOWGROwTH}) for a fixed $k$, this inequality is also satisfied
with the same $k$ for any $(Q,\lambda)$ in a neighbourhood of $(P,\mu)$. Moreover, by Theorem \ref{mainthm},
given any $(P,\mu)\in \mathbb R_d[t]\times\mathbb T$, we know that we may always find an integer $k$ such that (\ref{EQLEMSLOWGROwTH}) is true.
Define now
$$\mathbb R_d^{\pi}[t]=\left\{P=\sum_{j=0}^d \theta_j t^j;\ \theta_j\in[0,2\pi]\right\}.$$
$\mathbb R_d^\pi[t]$ is a compact subset of $\mathbb R_d[t]$. Moreover, for any $P\in\mathbb R_d[t]$, there
exists some $Q\in \mathbb R_d^\pi[t]$ such that $e^{iQ(k)}=e^{iP(k)}$ for any $k\in\mathbb Z$. Hence, Lemma \ref{LEMSLOWGROWTH}
follows from the compactness of $\mathbb R_d^\pi[t]\times\TT$.
\end{proof}
We come back to the proof of Theorem \ref{THMSLOWGROWTH}. We fix $\veps>0$ and $y\in X$. We apply Lemma \ref{LEMSLOWGROWTH}
to the 4-tuple $(d,x,y,\veps)$. We then set $\delta=\veps/\|T\|^K$. Since $x\in HC(T)$, there exist $n\geq 1$ as large as necessary
and $\alpha_n\in\TT$ such that
$$e^{if(n)}T^n x=\alpha_n x+z,\ \|z\|<\delta.$$
Then, there exists $k\leq K$ such that, setting $P_n(k)=\sum_{l=0}^d g_l(n)k^l$,
$$e^{iP_n(k)}T^k x=\alpha_n^{-1}y+z',\ \|z'\|<\veps.$$
Thus,
\begin{eqnarray*}
 e^{if(n+k)}T^{n+k}x&=&e^{i\big(f(n+k)-f(n)\big)}T^k\big(e^{if(n)}T^n x\big)\\
&=&e^{i\veps_k(n)}e^{iP_n(k)}T^k\big(\alpha_n x+z\big)\\
&=&e^{i\veps_k(n)}\big(y+\alpha_n z'+e^{iP_n(k)}T^k z\big).
\end{eqnarray*}
Since $\sup_{0\leq k\leq K}|\veps_k(n)|$ goes to 0 as $n$ goes to infinity, we get
$$\big\|e^{if(n+k)}T^{n+k}x-y\big\|<3\veps,$$
provided $n$ has been chosen large enough.
\end{proof}
This theorem covers the cases of many sequences which do not grow too quickly to infinity, like $f(n)=n^a\log^b(n+1)$, $(a,b)\in\mathbb R^2$, or finite linear
combinations of such functions.
Interestingly, we may also observe that Theorem \ref{THMPREP} does not extend to this level of generality.
\begin{example} \label{noTHMPREP}
Let $T=2B$ acting on $\ell^2(\mathbb Z_+)$. There exists $x\in HC(T)$ such that $(e_0,1)$ does not belong to
$\overline{\{(T^n x,e^{i\log (n)});\ n\geq 1\}}$.
\end{example}
\begin{proof}
Observe that
$$\Re e(e^{i\log (n)})\geq 0\iff n\in \bigcup_{k\geq 0}[\exp(2k\pi-\pi/2),\exp(2k\pi+\pi/2)]\iff n\in\bigcup_{k\geq 0} (a_k,b_k)$$
where $(a_k)$ and $(b_k)$ are sequences of integers satisfying $a_{k+1}-b_k\to+\infty$. Let $(x_k)$ be a dense subset of $\ell^2(\mathbb Z_+)$ such that
for any $k\geq 1$, $\|x_k\|\leq k$ and $x_k$ has a finite support contained in $[0,a_{k+1}-b_k-1]$. We set
$$x=\sum_{j\geq 1}\frac1{2^{b_{j^2}}}S^{b_{j^2}}x_j,$$
where $S$ is the forward shift. Provided $n\in \bigcup_k (a_k,b_k)$, we know that $\|T^n x-e_0\|\geq 1$ since $\langle T^n x,e_0\rangle=0$. In particular,
$(e_0,1)$ does not belong to $\overline{\{(T^n x,e^{i\log (n)});\ n\geq 1\}}$. Nevertheless, $x$ is a hypercyclic vector for $T$. Indeed,
$$\|T^{b_{k^2}}x-x_k\|\leq \sum_{j\geq k+1} \left\|\frac{1}{2^{b_{j^2}-b_{k^2}}}S^{b_{j^2}-b_{k^2}}x_j\right\|\leq \sum_{j\geq k+1}\frac{j}{2^{b_{j^2}-b_{k^2}}}\leq \sum_{j\geq k+1}\frac{j}{2^j}\xrightarrow{k\to+\infty} 0.$$
\end{proof}

As a consequence of Theorem \ref{THMSLOWGROWTH}, given $T\in \LX$ and a sequence of real numbers $(f(n))$
such that $f(n+1)-f(n) \to 0$, then $x\in HC(T)$ if and only if $\{e^{2\pi i f(n)}T^n x;\ n\geq
1\}$ is dense in $X$. Rephrasing this, we can say that for every sequence $(\lambda_n)\subset \mathbb{T}$ with
$\lambda_{n+1}/\lambda_n \to 1$, $x\in HC(T)$ if and only if $\{ \lambda_nT^n x;\ n\geq 1\}$ is dense in $X$.
One now may wonder whether the assumption ``$(\lambda_n)\subset \mathbb{T}$" can be dropped. However, this is
not the case. Indeed, Le\'{o}n-Saavedra proved in \cite{L} that there exists a hypercyclic operator $T\in \LX$
such that for every $x\in X$ the set $\{ \lambda_nT^nx;\ n\geq 1\}$ is not dense in $X$ where $\lambda_n=1/n$,
$n=1,2,\ldots $ and of course $\lambda_{n+1}/\lambda_n\to 1$. In general, if we keep the assumption
$\lambda_{n+1}/\lambda_n\to 1$ but we move away form the unit circle $\mathbb{T}$, that is $\lambda_n\in
\mathbb{C}$, then none of the implications $(ii)\implies (i)$, $(iii)\implies (ii)$ in Theorem \ref{mainthm}
hold. This follows by Propositions 2.4 and 2.5 from \cite{CH2}. On the other hand, under the additional
assumption that $T\in \LX$ is hypercyclic it is known \cite{CH1} that for $(\lambda_n)$ a sequence of non-zero
complex numbers with $\lambda_{n+1}/\lambda_n\to 1$ and $x\in X$, if the set $\{ \lambda_nT^nx;\ n\geq 1\}$ is
somewhere dense then it is everywhere dense.

\section{Uniformly distributed sequences and generic statements}
In Section 4 we showed that Theorem \ref{mainthm} is no longer true for sequences of unimodular complex numbers
whose phases grow to infinity at a geometric rate, see Proposition \ref{geomrate}. In particular, denoting by
$B$ the unweighted backward shift on $\ell^2(\mathbb Z_+)$ then, for every $x\in HC(2B)$ there exists
$\theta\in\mathbb R$ such that the set $\big\{ e^{2\pi i2^n\theta}(2B)^n x;\ n\geq 1\big\} $ is not dense in
$\ell^2(\mathbb Z_+)$. In this section we shall establish results, both in measure and category, going to the
opposite direction. For instance, a consequence of our result, Proposition \ref{prop1generic}, is that for every
$x\in HC(2B)$ the set of $\theta$'s in $\mathbb{R}$ such that $\big\{ e^{2\pi i2^n\theta }(2B)^n x;\ n\geq
1\big\} $ is dense in $\ell^2(\mathbb Z_+)$ is residual and of full Lebesgue measure in $\mathbb{R}$. This is in
sharp contrast with Proposition \ref{geomrate}. Such kind of behavior comes as a  natural consequence from
the general metric theorem of Koksma, see Theorem 4.3 in Chapter 1 of \cite{KN}. Koksma's theorem generalizes
the beautiful result of Weyl \cite{KN}: \textit{if $(n_k)$ is a distinct sequence of integers then the sequence
$(n_k\theta )$ is uniformly distributed for almost every $\theta \in \mathbb{R}$ }. Here, we shall need the
following corollary of Koksma's theorem, see Corollary 4.3 in Chapter 1 of \cite{KN}, which we state as a
theorem.

\begin{theorem} (Koksma) \label{Koksma}
Let $(f(n))$ be a sequence of real numbers such that for some $\delta>0$, $|f(n)-f(m)|>\delta$ for $n\neq m$. Then
the sequence $(f(n)\theta)$ is uniformly distributed for almost every $\theta \in \mathbb{R}$.
\end{theorem}

\begin{proposition} \label{prop1generic}
Let $X$ be a Banach space and let $T\in \LX$ be a hypercyclic operator. Let also $(f(n))$ be a sequence of real
numbers such that for some $\delta>0$, $|f(n)-f(m)|>\delta$ for $n\neq m$. Then for every $x\in HC(T)$ there
exists a residual subset $A$ of $\mathbb{R}$ with full measure such that for every $\theta \in A$ the set $\{
(T^nx,e^{2\pi i f(n)\theta });\ n\geq 1\}$ is dense in $ X\times \mathbb{T}$.
\end{proposition}
\begin{proof}
Take $x\in HC(T)$. Let $\{ x_j;\ j\in \mathbb{N} \}$, $\{ t_l;l\in \mathbb{N}\}$ be dense subsets of $X$ and
$\mathbb{T}$ respectively and define $A:=\bigcap_{j,l,s \in \mathbb{N}}\bigcup_{n\in \mathbb{N}}A_{j,l,s,n}$, where
$$A_{j,l,s,n}:=\{ \theta \in \mathbb{R};\ | e^{2\pi i f(n)\theta }-t_l|<1/s \,\,\, \textrm{and} \,\,\,
\| T^nx-x_j\| <1/s \} ,$$ for $j,l,s,n\in \mathbb{N}$. Clearly, for every $\theta \in A$ the set $\{
(T^nx,e^{2\pi i f(n)\theta });\ n\geq 1\}$ is dense in $ X\times \mathbb{T}$. Let us first prove that $A$ is
residual in $\mathbb{R}$. It is easy to see that $A_{j,l,s,n}$ is open in $\mathbb{R}$ and by Baire's category
theorem it remains to show that for every $j,l,s\in \mathbb{N}$ the set $\cup_{n\in \mathbb{N}}A_{j,l,s,n}$ is
dense in $\mathbb{R}$. To this end, fix $j,l,s\in \mathbb{N}$ and let $\alpha \in \mathbb{R}$, $\epsilon >0$.
Consider the following set of positive integers $ \{ n_1<n_2<\cdots \} :=\{ n\in \mathbb{N};\ \| T^nx-x_j\| <1/s
\}$. By Theorem \ref{Koksma} the sequence $(f(n_k)\theta )$ is uniformly distributed for almost every $\theta
\in \mathbb{R}$. In particular, there exists $\theta \in \mathbb{R}$ with $|\theta -\alpha |<\epsilon $ such
that the set $\{ e^{2\pi i f(n_k)\theta };\ k\in \mathbb{N} \}$ is dense in $\mathbb{T}$. Thus, $|\alpha-\theta
|<\epsilon $, $| e^{2\pi i f(n_m)\theta } -t_l|<1/s$ and $\| T^{n_m}x-x_j\| <1/s$ for some positive integer $m$.
This shows the density result. In the above, it is implicit that for $j,l,s\in \mathbb{N}$ the set $\bigcup_{n\in
\mathbb{N}}A_{j,l,s,n}$ has full measure. Hence, $A$ has full measure.
\end{proof}

 In view of the above, one
may be tempted to ask the following question.

\medskip

{\bf Question.} Let $X$ be a Banach space, let $T\in \LX $ be hypercyclic and let $x\in HC(T)$. Consider a sequence of real
numbers $(f(n))$ such that the sequence $(f(n)\theta )$ is uniformly distributed for some $\theta \in
\mathbb{R}$. Is it true that $\{ e^{2\pi if(n)}T^nx;\ n\geq 1\}$ is dense in $X$?

\medskip

Although in many cases the above question admits an affirmative reply, for instance when $f(n)$ is a polynomial
in $n$, the answer in general is no! To see this, fix a hypercyclic vector $x\in \ell^2(\mathbb Z_+)$ for
$2B\in \mathfrak L (\ell^2(\mathbb Z_+))$, where $B$ is the unweighted backward shift. By Proposition
\ref{geomrate} there exists $\theta '\in \mathbb{R}$ such that the set $\big\{ e^{2\pi i2^n\theta '}(2B)^n x;
n\geq 1\big\} $ is not dense in $\ell^2(\mathbb Z_+)$. Define now $f(n)=2^n\theta '$. Now, on the one hand, by
Theorem \ref{Koksma} there exists $\theta \in \mathbb{R}$ such that the sequence $(f(n)\theta )$ is uniformly
distributed and, on the other hand, $\big\{ e^{2\pi if(n)}(2B)^n x;\ n\geq 1\big\} $ is not dense in
$\ell^2(\mathbb Z_+)$. Regarding the previous question it is important to note that there are sequences of real
numbers $(f(n))$ such that for no $\theta \in \mathbb{R}$ the sequence $( f(n) \theta )$ is uniformly
distributed and yet $\overline{ \{ e^{2\pi if(n)}T^nx;\ n\geq 1\} }=X$ for every $x\in HC(T)$. This is the case
for the sequence $f(n)=\log (n)$, $n=1,2,\ldots $. Indeed, this follows by Theorem \ref{THMSLOWGROWTH} and
the fact that for every $\theta \in \mathbb{R}$ the sequence $( \log (n)\theta)$ is not uniformly distributed,
see Examples 2.4, 2.5 in Chapter 1 of \cite{KN}. It is clear now that the right question along this line is the
following.
\medskip

{\bf Question.} Let $X$ be a Banach space, let $T\in \LX $ be hypercyclic and let $x\in HC(T)$. Consider a
sequence of real numbers $(f(n))$ which is uniformly distributed. Is it true that the set $\{ e^{2\pi
if(n)}T^nx;\ n\geq 1\}$ is dense in $X$?

\medskip
In turn, the answer to this question is negative.
\begin{proposition}
 Let $T=2B$ acting on $X=\ell^2(\mathbb Z_+)$. There exists a uniformly distributed sequence of real numbers $(f(n))$
and a vector $x\in HC(T)$ such that $\{e^{2\pi i f(n)}T^n x;\ n\geq 1\}$ is not dense in $X$.
\end{proposition}
\begin{proof}
 We start from any uniformly distributed sequence $(g(n))$. The idea of the proof is to change slightly
the sequence $(g(n))$ to a sequence $(f(n))$ which remains uniformly distributed and such that $e^{2\pi if(n)}$ can be arbitrarily
chosen for $n$ belonging to some subset $A\subset \mathbb N$ containing arbitrarily large intervals.
We define simultaneously $x\in HC(T)$ such that
$$\left\{\begin{array}{rcll}
          e_0^*(T^n x)&=&0&\textrm{ provided }n\notin A\\
\Re e\big(e_0^*(e^{2\pi if(n)}T^n x)\big)&\geq&0&\textrm{ provided }n\in A.
         \end{array}
\right.$$
Hence, $\{e^{2\pi if(n)}T^n x;\ n\geq 1\}$ will not be dense in $X$.

We now proceed with the details. Let $(n_k)$ be an increasing sequence of integers such that
$n_1=1$, $n_{k+1}>n_k+(k+1)$ for any $k\geq 1$ and $k^2/n_k\to 0$. Let $(x_k)$ be a dense sequence in $\ell^2(\mathbb Z_+)$
such that, for any $k\geq 1$, $\|x_k\|\leq k$ and $x_k$ has finite support contained in $[0,k]$. We set
$$x=\sum_{j\geq 1}\frac1{2^{n_j}}S^{n_j}x_j$$
and we observe that $x\in HC(T)$. Indeed,
$$\|T^{n_k}x-x_k\|=\left\|\sum_{j>k}\frac1{2^{n_j-n_k}}S^{n_j-n_k}x_j\right\|\leq\sum_{j>k}\frac j{2^j}\xrightarrow{k\to+\infty}0.$$
We then define $(f(n))$ by
\begin{itemize}
 \item $f(n)=g(n)$ provided $n\notin\bigcup_{k\geq 1}[n_k,n_k+k]$;
\item $f(n)$ is any positive real number such that $\Re e\big(e_0^*(e^{2\pi if(n)}T^n x)\big)\geq0$
provided $n\in \bigcup_{k\geq 1}[n_k,\ n_k+k]$.
\end{itemize}
As already observed, $\{e^{2\pi if(n)}T^n x;\ n\geq 1\}$ cannot be dense in $\ell^2(\mathbb Z_+)$. Thus it remains to
show that $(f(n))$ is uniformly distributed. Let $I$ be a subarc of $\TT$ and let $n\geq 1$. Let $k\geq 1$ be
such that $n_k\leq n<n_{k+1}$. Then, since $f(j)$ and $g(j)$ may only differ for $j\leq n$ if $j$ belongs to
$\bigcup_{l=1}^k[n_l,n_l+l]$ which has cardinality less than $(k+1)^2$,
$$-\frac{(k+1)^2}n+\frac1n\textrm{card}\left(\left\{1\leq j\leq n;\ e^{2\pi ig(j)}\in I\right\}\right)\leq \frac1n\textrm{card}\left(\left\{1\leq j\leq n;\ e^{2\pi if(j)}\in I\right\}\right)$$
and
$$\frac1n\textrm{card}\left(\left\{1\leq j\leq n;\ e^{2\pi if(j)}\in I\right\}\right)\leq \frac1n\textrm{card}\left(\left\{1\leq j\leq n;\ e^{2\pi ig(j)}\in I\right\}\right)+\frac{(k+1)^2}n.$$
Since $k^2/n\leq k^2/n_k$ goes to zero, $(f(n))$ remains uniformly distributed.
\end{proof}

\medskip

We conclude the paper with a generic result which is related to the question we asked in the introduction. The space
$\mathbb{T}^{\mathbb{N}}$ is endowed with the metric $d(\Lambda
,M)=\sum_{n=1}^{+\infty}\frac{|\lambda_n-\mu_n|}{2^n}$ for $\Lambda=(\lambda_n)\in \mathbb{T}^{\mathbb{N}}$,
$M=(\mu_n)\in \mathbb{T}^{\mathbb{N}}$, and becomes a complete metric space.
\begin{proposition} \label{prop2generic}
Let $X$ be a Banach space and let $T\in \LX$ be a hypercyclic operator. Then for every $x\in HC(T)$ there exists
a residual subset $B$ of $\mathbb{T}^{\mathbb{N}}$ such that for every $(\lambda_1, \lambda_2,\ldots )\in B$ the
set $\{ \lambda_nT^nx;\ n\geq 1\}$ is dense in $X$.
\end{proposition}
\begin{proof}
Fix $x\in HC(T)$. Let $\{ x_j;\ j\in \mathbb{N}\}$ be a dense set in $X$ and define the set $B_{j,s,n}:=\{
\Lambda=(\lambda_m)\in \mathbb{T}^{\mathbb{N}};\ \| \lambda_nT^nx-x_j\| <1/s\}$ for $j,s,n\in \mathbb{N}$. It is
straightforward to check that $B_{j,s,n}$ is open in $\mathbb{T}^{\mathbb{N}}$ for every $j,s,n\in \mathbb{N}$.
We then define $B:=\bigcap_{j,s \in \mathbb{N}}\bigcup_{n\in \mathbb{N}}B_{j,s,n}$. Observe that if $\Lambda
=(\lambda_n)\in B$ then the set $\{ \lambda_nT^nx;\ n\geq 1\}$ is dense in $X$. Hence, in view of Baire's
category theorem it suffices to show that for every $j,s\in \mathbb{N}$ the set $\bigcup_{n\in
\mathbb{N}}B_{j,s,n}$ is dense in $\mathbb{T}^{\mathbb{N}}$. Fix $j,s\in \mathbb{N}$ and let $M=(\mu_n)\in
\mathbb{T}^{\mathbb{N}}$, $\epsilon >0$. There exists a positive integer $N$ such that
$\sum_{n=N+1}^{+\infty}\frac{2}{2^n}<\epsilon$. Define the vector $\Lambda:=(\mu_1, \ldots, \mu_N, 1,1,1, \ldots
)\in \mathbb{T}^{\mathbb{N}}$. From the above we get $d(\Lambda ,M)<\epsilon$ and $\| \lambda_nT^nx-x_j\| =\|
T^nx-x_j\| <1/s$ for some $n\in \mathbb{N}$ with $n >N$. Hence, the set $\bigcup_{n\in \mathbb{N}}B_{j,s,n}$ is
dense in $\mathbb{T}^{\mathbb{N}}$.
\end{proof}

\end{document}